\documentclass [10pt]{article}

\usepackage{amsfonts}
\usepackage{graphicx}
\usepackage{amsmath}

\newtheorem{theorem}{Theorem}[section]

\newtheorem{assumption}[theorem]{Assumption}

\newtheorem{claim}[theorem]{Claim}

\newtheorem{definition}[theorem]{Definition}

\newtheorem{remark}[theorem]{Remark}

\newenvironment{proof}[1][Proof]{\noindent\textit{#1.} }{\hfill \rule{0.5em}{0.5em}}

\newcommand{\R}{\mathbb{R}}
\renewcommand{\d}{{\rm d}}
\renewcommand{\eqref}[1]{(\ref{#1})}

\begin{document}

\title{\textbf{Population invasion with bistable dynamics and adaptive evolution: the evolutionary rescue}}
\author{
\textsc{Matthieu Alfaro}\\
{\small \textit{IMAG, Universit\'e de
Montpellier,}} \\
{\small \textit{CC051, 34095 Montpellier
Cedex 5, France.}}\\
{\small \textit{email: matthieu.alfaro@umontpellier.fr}}\\
\textsc{Arnaud Ducrot}\\
{\small \textit{Univ. Bordeaux, IMB, UMR 5251, F-33400 Talence, France}} \\
		{\small \textit{CNRS, IMB, UMR 5251, F-33400 Talence, France.}}\\
		{\small \textit{email: arnaud.ducrot@u-bordeaux.fr}} 
}
\maketitle

\begin{abstract} We consider the system of reaction-diffusion equations proposed in \cite{Kanarek-Webb} as a population dynamics model. The first equation 
stands for the population density and models the ecological effects, namely dispersion and growth with a {\it Allee effect} (bistable nonlinearity). The second one stands for the Allee threshold, seen as a trait mean, and accounts for evolutionary effects. Precisely, the Allee threshold is submitted to three main effects: dispersion (mirroring ecology), asymmetrical gene flow and selection. The strength of the latter depends  on the population density and is thus coupling ecology and evolution. Our main result is to mathematically prove {\it evolutionary rescue}: any small initial population, that would become extinct in the sole ecological context, will persist and spread thanks to evolutionary factors.

\vspace{0.2in}\noindent \textbf{Key words}. Reaction-diffusion system, Allee effect, long time behaviour, energy method, evolutionary rescue.

\vspace{0.1in}\noindent \textbf{2000 Mathematical Subject Classification}. 35K45, 92B05, 92D15.
\end{abstract}


\section{Introduction}\label{s:intro}

In this work we consider the following reaction-diffusion system
\begin{equation}\label{pb}
\begin{cases}
\displaystyle \frac{\partial u}{\partial t}=\frac{\partial^2 u}{\partial x^2}+u\left(u-a^2\right)(1-u),&\;t>0,\;x\in\R,\vspace{5pt}\\
\displaystyle \frac{\partial a}{\partial t}=\frac{\partial^2 a}{\partial x^2}+2\frac{\partial a}{\partial x}\frac{\partial \ln(u)}{\partial x}-\varepsilon (1-u)a,&\;t>0,\;x\in\R,
\end{cases}
\end{equation}
where $\varepsilon>0$ is a given parameter. The above system is supplemented with an initial  data
\begin{equation}\label{pb1}
\begin{cases}
u(0,x)=u_0(x)& \text{ with $u_0\in C(\R)$, $0\leq u_0\leq 1$ and $u_0\not\equiv 0$,}\\
a(0,x)=a_0& \text{ with $a_0\in (0,1)$}.
\end{cases}
\end{equation} 
This system of equations was proposed by Kanarek and Webb in \cite{Kanarek-Webb} to analyze the effect of environmental adaptive evolution on a species persistence or invasion.

In the above system, $u=u(t,x)$ denotes the density of an invasive species at time $t\geq 0$ and spatial location $x\in\R$. The first term in the right-hand side of the first equation stands for spatial diffusion, whereas the second one stands for the growth of the population which is assumed to exhibit an {\it Allee effect} with threshold $a^2$ (the growth per capita is negative when the population density is below this threshold).  Typically $u_0$, the introduced amount of population, is small and spatially localized, say compactly supported. 

In \eqref{pb}, the Allee threshold $a^2$ is assumed to be a spatio-temporal varying parameter. It is considered as a mean fitness evolution trait. Its evolution is ruled by the second equation in \eqref{pb}. The first term in the right-hand side mimics the diffusion of the population while the second term describes the {\it gene flow} due to the population gradient. Thus, these two terms take into account the joint influence of the motion and the position of individuals on the mean trait value $a$. We refer the reader to the works of Pease et al. \cite{Pease}, Kirkpatrick and Barton \cite{Kirkpatrick}, Garc\'{i}a-Ramos and Kirkpatrick \cite{Garcia}  for modelling details on the convection gene flow term. The last term in the right-hand side denotes the selection gradient with genetic variance parameter $\varepsilon>0$, that is typically small, and couples ecology and evolution.  We refer to \cite{Kanarek-Webb} for more details and explanations on the modelling issue.

When the genetic variance parameter $\varepsilon$ is null then $a(t,x)\equiv a_0$ is no longer a dynamical parameter, and the model reduces to the single bistable equation
\begin{equation*}\label{bistable}
\displaystyle \frac{\partial u}{\partial t}=\frac{\partial^2 u}{\partial x^2}+u\left(u-a_0^2\right)(1-u),
\end{equation*}
with initial data $u(0,x)=u_0(x)$. For such a problem, spatial invasion strongly depends upon both the initial data $u_0$ and the Allee threshold $a_0^2$.
For instance if $0\leq u_0\leq a_0^2$ is compactly supported then the solution $u=u(t,x)$ converges to $0$ as $t\to+\infty$, uniformly for $x\in\R$. Hence a localized small amount of introduced population leads to extinction and the population invasion fails. Even if $\max_{x\in\R}u_0(x)>a_0^2$ then population invasion may fail, in particular if the set $\{x\in\R:\;u_0(x)>a_0^2\}$ is somehow too small. We refer to the works of Fife and McLeod \cite{Fife-Leod}, Zlat\v{o}s \cite{Zlatos}, Du and Matano \cite{Du-Matano}, Muratov and Zhong \cite{Muratov-Zhong} for precise results on the so-called sharp threshold condition for bistable, combustion type equations and also for more general scalar reaction-diffusion equations. See also the work of Pol\'{a}\v{c}ick \cite{Polacik} for related results for non-autonomous equations.
This threshold phenomenon is in sharp contrast with Fisher-KPP dynamics, e.g. $f(u)=u(1-u)$, for which any small amount of population implies successful invasion. This is referred as to the
{\it Hair Trigger Effect}. For more general monostable dynamics, e.g. $f(u)=u^{p}(1-u)$ with $p>1$, we refer to Aronson and Weinberger \cite{Aronson-Weinberger}: the Hair Trigger Effect --- which is related to the {\it Fujita blow-up phenomenon} \cite{Fujita}, \cite{Alfaro}--- holds if and only if $p\leq 3$ (in dimension one).

Hence, the threshold phenomenon mentioned above holds true for system \eqref{pb} without mutation, that is in the extreme case $\varepsilon=0$. However based on numerical simulations, Kanarek and Webb \cite{Kanarek-Webb} show that some solutions of system \eqref{pb}, that would go extinct for $\varepsilon=0$, exhibit successful invasion as soon as $\varepsilon>0$. In other words, according to  \cite{Kanarek-Webb}, adaptive evolution may save species from the brink of extinction and may enable successful invasions. This phenomenon is referred to as {\it Evolutionary rescue}.
The aim of this note is to rigorously prove this statement in the context of the model \eqref{pb}--\eqref{pb1}.   

\section{Main results}\label{s:results}

Throughout this note, a solution $(u,a)=(u(t,x),a(t,x))$ for the initial value problem \eqref{pb}--\eqref{pb1} is understood in the classical sense, which is more precisely stated in the following definition. 

\begin{definition}\label{def-sol}
A function pair $(u,a)=(u(t,x),a(t,x))$ is said to be a solution of the initial value problem \eqref{pb}--\eqref{pb1} if it satisfies the following set of properties.
\begin{itemize}
\item[(i)] The functions $u$ and $a$ belong to $C^{1,2}\left((0,+\infty)\times \R\right)$ and $0<u(t,x)\leq 1$, $0<a(t,x)\leq a_0$, for all $t>0$, $x\in \R$.
\item[(ii)] The pair $(u,a)$ satisfies \eqref{pb} for all $t>0$ and $x\in\R$.
\item[(iii)] $\left(u(t,x),a(t,x)\right)\to \left(u_0(x),a_0\right)$, locally uniformly for $x\in\R$, as $t\to 0^+$.
\end{itemize}
\end{definition}

Our first result is concerned with the existence of a solution for \eqref{pb}--\eqref{pb1}.

\begin{theorem}[Existence of a solution]\label{theo-existence}
Let $(u_0,a_0)\in C(\R)\times (0,1)$ be a given initial data as in \eqref{pb1}.
Then  \eqref{pb} supplemented with the initial data $(u_0,a_0)$ admits, at least, a solution $(u,a)=\left(u(t,x),a(t,x)\right)$ in the sense of Definition \ref{def-sol}, that moreover satisfies
\begin{equation*}
\sup_{x\in\R}\left|a(t,x)-a_0\right|=O\left(t\right),\text{ as $t\to 0^+$}.
\end{equation*}
\end{theorem}

The main difficulty in proving the above theorem consists in handling the singular convection term in the second equation in \eqref{pb}. Indeed, while the function $u$ becomes immediately positive, it is rather intricate to handle the term $\partial_x\ln(u)$, especially for small times when $u_0$ vanishes at some places. The above result is proved using a regularisation procedure and a suitable limiting argument to recover the initial data. As it will be clear from the proof given below, this procedure only allows to handle constant initial data $a(0,x)\equiv a_0$. The case of non-constant initial data $a(0,x)$ is much more delicate and is not considered here. One may also note that the above result provides the existence of a solution but not the uniqueness, that is also an open problem especially when $u_0$ vanishes at some locations.

Our next result deals with the asymptotic behaviour as $t\to+\infty$ of the solutions of the Cauchy problem \eqref{pb}--\eqref{pb1}.

\begin{theorem}[Evolutionary rescue result]\label{theo-asymp} Let $(u_0,a_0)\in C(\R)\times (0,1)$ be a given initial data as in \eqref{pb1}.
Let $(u,a)=\left(u(t,x),a(t,x)\right)$ be a solution of \eqref{pb}--\eqref{pb1}. Then the population density $u(t,x)$ enjoys the following weak persistence property 
\begin{equation*}
\limsup_{t\to+\infty}\sup_{x\in\R} u(t,x)=1.
\end{equation*}
If we furthermore assume that $a_0^2<\frac{1}{2}$, then the population exhibits a total spatial invasion, in the sense that
\begin{equation*}
\lim_{t\to+\infty} u(t,x)=1\text{ locally uniformly for $x\in\R$}.
\end{equation*}
\end{theorem}

In the above theorem, the first part ensures that the population does not uniformly become extinct at large times. More interestingly when the initial Allee threshold is small, namely $a_0^2<\frac{1}{2}$, the population persists and successfully invades the whole space. Note that this assumption is hardly a restriction since in ecological problems, the Allee threshold is usually rather small and in particular typically smaller than $\frac{1}{2}$. This is the case in the simulations of \cite{Kanarek-Webb} where $a_0^2$ is set to $0.3$. Notice also that $a_0^2<\frac 12$ is equivalent to $\int _0 ^1 u(u-a_0^2)(1-u)\d u>0$, which is a natural assumption  in the seminal work of Fife and McLeod \cite{Fife-Leod} to allow an invasion  in a bistable situation.

To prove the above theorem, we first derive preliminary results for a single bistable reaction-diffusion equation with time varying threshold $\theta(t)$, namely   
\begin{equation*}
\frac{\partial u}{\partial t}=\frac{\partial^2 u}{\partial x^2}+u\left(u-\theta(t)\right)(1-u), \;\;t>0,\;x\in\R.
\end{equation*}
Using a suitable energy method, we roughly prove that when the Allee threshold $\theta(t)$ decay to zero rather fast as $t\to+\infty$ (see Assumption \ref{ASS-threshold} for a precise statement), then the above bistable equation enjoys the Hair Trigger Effect. This result is the goal of Section \ref{s:varying}, while both Theorem \ref{theo-existence} and Theorem \ref{theo-asymp} are proved in Section \ref{s:rescue}. 

\section{Bistable equations with decreasing threshold}\label{s:varying}

In this section we provide preliminary results that will be used to study \eqref{pb}. We consider the bistable reaction-diffusion equation with time varying threshold
\begin{equation}\label{eq1}
\frac{\partial u}{\partial t}=\frac{\partial^2 u}{\partial x^2}+f\left(\theta(t),u\right),\;t>0,\;x\in\R.
\end{equation}
The function $f=f(\theta,u)$ reads as the following bistable nonlinearity with threshold $\theta\in (0,1)$
\begin{equation}\label{function-f}
f(\theta,u)=u(u-\theta)(1-u).
\end{equation}
Here we consider that the threshold $\theta$ depends on time, $\theta=\theta(t)$, and decreases to $0$ sufficiently fast as $t\to+\infty$. Our precise set of assumptions reads as follows.

\begin{assumption}[Decreasing threshold]\label{ASS-threshold}
The function $\theta:[0,+\infty)\mapsto (0,1)$ is continuous, decreasing and satisfies
\begin{equation*}
\int_0^{+\infty} \theta(s)\d s<+\infty.
\end{equation*}
\end{assumption}

Under the above assumption, if $u_0\in C(\R)$ is a given function such that $0\leq u_0\leq 1$, $u_0\not\equiv 0$ and $u_0\not\equiv 1$, then \eqref{eq1} supplemented with the initial data $u(0,x)=u_0(x)$ admits a unique solution $u=u(t,x)$ that satisfies $0<u(t,x)<1$ for all $t>0$ and $x\in\R$. The main result of this section is concerned with its asymptotic behaviour as $t\to+\infty$.   

Note that bistable reaction-diffusion equations with time-varying threshold have been considered by Pol\'{a}\v{c}ick in \cite{Polacik}. In this work, the author deals with varying threshold that stays uniformly away from $0$ and $1$ and proves that such a problem exhibits a sharp threshold behaviour as in the case of a fixed threshold.
Here our problem is very different since the threshold function $\theta(t)$ decays to zero. In contrast with the sharp threshold effect proved by Pol\'{a}\v{c}ick in the aforementioned paper, we prove that the Hair Trigger Effect holds for \eqref{eq1} under Assumption \ref{ASS-threshold}. This reads as follows.

\begin{theorem}[Hair Trigger Effect]\label{theorem1}
Under Assumption \ref{ASS-threshold}, for each $u_0\in C(\R)$ with $0\leq u_0\leq 1$ and $u_0\not\equiv 0$, the solution $u=u(t,x)$ of \eqref{eq1} with initial data $u_0$ satisfies
\begin{equation*}
\lim_{t\to+\infty} u(t,x)=1\; \text{ locally uniformly for $x\in\R$}.
\end{equation*}
\end{theorem}

To prove this theorem we use an energy method for a similar problem posed on a bounded interval $(-R,R)$ supplemented with Dirichlet boundary conditions.
\medskip

\begin{proof}[A preliminary observation] Let $R>0$ and $\theta_0\in (0,1)$ be given.
Consider the Dirichlet problem 
\begin{equation}\label{eq-Dirichlet}
\begin{cases}
\displaystyle \frac{\partial v}{\partial t}=\frac{\partial^2 v}{\partial x^2}+f\left(\theta_0,v\right),&\;t>0,\;x\in (-R,R),\vspace{5pt}\\
\displaystyle v(t,\pm R)=0,&\;t>0,
\end{cases}
\end{equation}
supplemented with an initial data $v_0\in H^1_0(-R,R)$ such that $0\leq v_0\leq 1$.
We consider the energy functional $\mathcal E_{R,\theta_0}=\mathcal E_{R}^{\rm d}+\mathcal E_{R,\theta_0}^{\rm r}$ defined on $H_0^1(-R,R)$ by
$$
\mathcal E_R^{\rm d}(\varphi)=\frac{1}{2}\int_{-R}^R \vert \partial _x \varphi(x)\vert ^2 \d x,
$$
and
$$
\mathcal E_{R,\theta_0}^{\rm r}(\varphi)=-\int_{-R}^R F(\theta_0,\varphi(x))\d x,
$$
wherein we have set $F(\theta_0, v):=\int_0^v f(\theta_0,s)\d s=-\frac 1 4 v^4+\frac{1+\theta_0}{3}v^3-\frac \theta 2 v^2$. 

Let us observe that if $v=v(t,x)$ is a solution of \eqref{eq-Dirichlet} then one has
\begin{equation*}
\frac{\d}{\d t}\mathcal E_{R,\theta_0}\left(v(t,\cdot)\right)=-\int_{-R}^R |\partial_t v(t,x)|^2\d x\leq 0,\;\forall t>0.
\end{equation*}
Notice that $\mathcal E_{R,\theta_0}(0)=0$ and $\mathcal E_{R,\theta_0}(v(t,\cdot))\to \mathcal E_{R,\theta_0}(v_0)$ as $t\to 0$. Hence, if $\mathcal E_{R,\theta_0}(v_0)<0$ then the solution $v=v(t,x)$ of \eqref{eq-Dirichlet} with the initial data $v_0$ has to satisfy
$$
\limsup_{t\to + \infty}\sup_{x\in [-R,R]}v(t,x)>0.
$$
In other words, if $\mathcal E_{R,\theta_0}(v_0)<0$ with $v_0\in H_0^1(-R,R)$ then the corresponding solution $v$ does not go to extinction as $t\to+\infty$.
\end{proof}

\medskip
Keeping the above argument in mind, we now come back to \eqref{eq1} and to Theorem \ref{theorem1}.
\medskip

\begin{proof}[Proof of Theorem \ref{theorem1}]
Thanks to the comparison principle, it is sufficient to consider  a compactly supported  initial data $u_0\in C(\R)$ with $0\leq u_0\leq 1$ and $u_0\not\equiv 0$. Denote by $u=u(t,x)$ the solution of \eqref{eq1} starting from $u_0$. We split our proof into two steps. We first show that $u$ does not go extinct as $t\to+\infty$ by using energy functional method. Next we show that that $u$ does converge to $1$ as $t\to+\infty$, locally uniformly in space. 

\noindent{\bf First step: $u$ does not go extinct as $t\to+\infty$.} The aim of this first step is to show that
\begin{equation}\label{1st}
\exists x_0\in\R,\;\;\limsup_{t\to+\infty} u(t,x_0)>0.
\end{equation}
First observe that 
\begin{equation*}
f(\theta(t),u)\geq -\theta(t)u\; \text{ for all $u\in [0,1]$, $t\geq 0$}.
\end{equation*}
Then the parabolic comparison principle applies and yields
\begin{equation}\label{comparaison}
u(t,x)\geq e^{-\int_0^t \theta(s)\d s} w(t,x)\geq \Theta w(t,x),\;\;\forall t>0,\;x\in\R,
\end{equation}
with $\Theta:=e^{-\int_0^{+\infty} \theta(s)\d s}>0$ (see Assumption \ref{ASS-threshold}), and
$w(t,x)$  the solution of the heat equation starting from $u_0$, that is
\begin{equation*}
\frac{\partial w}{\partial t}=\frac{\partial^2 w}{\partial x^2},\;t>0,\;x\in\R;\quad w(0,\cdot)=u_0.
\end{equation*}
Introducing the heat kernel $\Gamma=\Gamma(t,x)$ defined by
\begin{equation*}
\Gamma(t,x)=\frac{1}{\sqrt{4\pi t}}e^{-\frac{x^2}{4t}},
\end{equation*}
the function $w$ re-writes as
\begin{equation*}
w(t,x)=(\Gamma(t,\cdot)*u_0)(x)=\int_{\R} \Gamma(t,x-y)u_0(y)\d y,\;t>0,\;x\in\R.
\end{equation*}
Now in order to prove that $u$ persists as $t\to+\infty$, we consider a smooth cut-off function $\chi:\R\to \R$ such that
\begin{equation*}
0\leq \chi\leq 1\; \text{ and }\; \chi(x)=\begin{cases} 1 & \text{ if }|x|\leq\frac{1}{2},\\ 0 & \text{ if }|x|\geq 1.\end{cases}
\end{equation*}
\begin{claim}\label{claim1}
There exists $t_0>0$ large enough such that the function 
\begin{equation*}\label{z_0}
v_0:=\Theta \chi\left(\frac{.}{\sqrt{t_0}}\right)w(t_0,\cdot)\in H^1_0\left(-\sqrt{t_0},\sqrt{t_0}\right)
\end{equation*}
satisfies $
\mathcal E_{\sqrt{t_0},\theta(t_0)}\left(v_0\right)<0$. 
\end{claim}
Equipped with the above claim, whose proof is postponed, we can conclude this first step. 
Indeed, note that since $\chi\leq 1$, \eqref{comparaison} ensures that $v_0\leq u(t_0,\cdot)$ on $[-R_0,R_0]$ with $R_0=\sqrt{t_0}$. Furthermore, since $t\mapsto \theta(t)$ is decreasing, one also has, for all $t\geq t_0$ and $x\in [-R_0,R_0]$,
$$
0=\frac{\partial u}{\partial t}-\frac{\partial^2 u}{\partial x^2}-f\left(\theta(t),u\right)
\leq \frac{\partial u}{\partial t}-\frac{\partial^2 u}{\partial x^2}-f\left(\theta(t_0),u\right).
$$
From the comparison principle, one gets
\begin{equation*}
u(t_0+t,x)\geq v(t,x),\;\forall t\geq 0,\;x\in (-R_0,R_0),
\end{equation*}
where $v$ is the solution of the Dirichlet problem \eqref{eq-Dirichlet}, with $\theta _0\leftarrow \theta(t_0)$ and $R\leftarrow R_0$, starting from $v_0$.
Since $\mathcal E_{R_0,\theta(t_0)}(v_0)<0$, we deduce from our preliminary observation that
\begin{equation*}
0<\limsup_{t\to+\infty} \sup_{x\in[-R_0,R_0]}v(t,x)\leq \limsup_{t\to+\infty} \sup_{x\in[-R_0,R_0]}u(t,x),
\end{equation*}
which concludes the proof of \eqref{1st}. To complete this first step, it remains to prove Claim \ref{claim1}.

\begin{proof}[Proof of Claim \ref{claim1}]
Set, for $t>0$ and $x\in\R$,
\begin{equation*}
z(t,x):=\Theta \chi\left(\frac{x}{\sqrt{t}}\right)w(t,x).
\end{equation*}

We first estimate $\mathcal E_{\sqrt t}^{\rm d}(z(t,\cdot))=\frac 12\left\|\partial_x z(t,\cdot)\right\|^2_{L^2(-\sqrt t,\sqrt t)}=\frac 12\left\|\partial_x z(t,\cdot)\right\|^2_{L^2(\R)}$.
To that aim, note that
\begin{equation*}
\frac{\partial z}{\partial x}(t,x)=\frac{\Theta}{\sqrt t}\chi'\left(\frac{x}{\sqrt t}\right) \left( \Gamma(t,\cdot)\ast u_0\right)(x)+ \Theta \chi\left(\frac{x}{\sqrt{t}}\right)\left(\frac{\partial \Gamma}{\partial x}(t,\cdot)\ast u_0\right)(x),
\end{equation*}
so that, for all $t>0$,
\begin{equation*}
\begin{split}
\left\|\frac{\partial z}{\partial x}(t,\cdot)\right\|_{L^2(\R)}\leq &\frac{\Theta \|\chi'\|_\infty}{\sqrt t}\|\Gamma(t,\cdot)\|_{L^2(\R)}\|u_0\|_{L^1(\R)}\\
&+ \Theta \left\|\frac{\partial \Gamma}{\partial x}(t,\cdot)\right\|_{L^2(\R)}\| u_0\|_{L^1(\R)},
\end{split}
\end{equation*}
and thus
\begin{equation*}
\left\|\frac{\partial z}{\partial x}(t,\cdot)\right\|_{L^2(\R)}=O\left(t^{-\frac{3}{4}}\right)\text{ as $t\to+\infty$}.
\end{equation*}
As a consequence we obtain
\begin{equation}\label{esti1}
\mathcal E_{\sqrt t}^{\rm d}(z(t,\cdot))=O\left(t^{-\frac{3}{2}}\right)\text{ as $t\to+\infty$}.
\end{equation}

We now estimate $\mathcal E^{\rm r}_{\sqrt t,\theta(t)}\left(z(t,\cdot)\right)$. This part of the energy is defined by
\begin{eqnarray*}
\mathcal E^{\rm r}_{\sqrt t,\theta(t)}\left(z(t,\cdot)\right)&=&\int_{-\sqrt t}^{\sqrt t}\left(\frac 1 4z^4(t,x)-\frac{1+\theta(t)}{3}z^3(t,x)+\frac{\theta(t)}{2}z^2(t,x)\right)
\d x\\
&\leq &\int_{-\sqrt t}^{\sqrt t}\left(\frac 1 4z^4(t,x)-\frac{1}{3}z^3(t,x)\right)\d x\\
&&+\frac{\theta(t)}{2}\Theta ^{2}\int_{-\sqrt t}^{\sqrt t}w^{2}(t,x)\d x.
\end{eqnarray*}
Now observe that $\|z(t,\cdot)\|_\infty=O\left(t^{-\frac{1}{2}}\right)$ as $t\to+\infty$. Hence, for $t>0$ large enough, we have
$$
\mathcal E^{\rm r}_{\sqrt t,\theta(t)}\left(z(t,\cdot)\right)\leq -\frac 1 6\int_{-\sqrt t}^{\sqrt t}z^3(t,x)\d x+\frac{\theta(t)}{2}\Theta ^{2}\int_{-\sqrt t}^{\sqrt t}w^{2}(t,x)\d x.
$$
Since $\chi(x)=1$ for $|x|\leq \frac{1}{2}$ we get, for $t>0$ large enough,
\begin{eqnarray}
\mathcal E^{\rm r}_{\sqrt t,\theta(t)}\left(z(t,\cdot)\right)&\leq& -\frac {\Theta ^{3}} 6\int_{-\frac{\sqrt t}{2}}^{\frac{\sqrt t}{2}}w^3(t,x)\d x+\frac{\theta(t)}{2}\Theta ^{2}\int_{-\sqrt t}^{\sqrt t}w^{2}(t,x)\d x\nonumber \\
&=& -\frac {\Theta ^{3}} 6\int_{-\frac{\sqrt t}{2}}^{\frac{\sqrt t}{2}}w^3(t,x)\d x+\theta(t) O(t^{-\frac 12}),
\label{truc}
\end{eqnarray}
since $\Vert w(t,\cdot)\Vert _\infty =O(t^{-\frac 12})$. Now recall that the function $w=w(t,x)$, the solution of the heat equation, becomes asymptotically self-similar in the sense that, for any $1\leq p\leq +\infty$, one has (see for instance  the monograph of Giga, Giga and Saal \cite[subsection 1.1.5]{Giga})
\begin{equation}\label{bidule}
\|w(t,\cdot)-\alpha\Gamma(t,\cdot)\|_{L^p(\R)}=o\left(t^{-\frac{1}{2}\left(1-\frac{1}{p}\right)}\right)\;\text{ as $t\to+\infty$},
\end{equation} 
wherein we have set $\alpha=\int_{\R} w(0,x)\d x>0$. Also, for any $1\leq p\leq +\infty$, there is $c_p>0$ such that
\begin{equation}\label{bidule2}
\left\|\Gamma(t,\cdot)\right\|_{L^p(\R)}=c_p t^{-\frac{1}{2}\left(1-\frac{1}{p}\right)}.
\end{equation}
Hence, denoting $\Vert \cdot\Vert _p=\Vert \cdot \Vert _{L^p(\R)}$, we have, as $t\to +\infty$,
\begin{eqnarray}
\int_{-\frac{\sqrt t}{2}}^{\frac{\sqrt t}{2}}w^3(t,x)\d x&=&O\Big(\Vert w(t,\cdot)-\alpha \Gamma(t,\cdot) \Vert _{3}^{3}+\Vert w(t,\cdot)-\alpha \Gamma(t,\cdot) \Vert _{2}^{2}\Vert \Gamma(t,\cdot) \Vert _{\infty}\nonumber \\
&&+\Vert w(t,\cdot)-\alpha \Gamma(t,\cdot) \Vert _{1}\Vert \Gamma(t,\cdot) \Vert _{\infty}^{2}\Big)+\alpha^{3}\int_{-\frac{\sqrt t}{2}}^{\frac{\sqrt t}{2}}\Gamma ^3(t,x)\d x\nonumber \\
&=&o\left(\frac 1 t\right)+\frac 1 t \left(\frac{\alpha}{\sqrt{4\pi}}\right)^{3}\int_{-\frac 12}^{\frac 12}e^{-\frac 34 y^{2}}\d y.\label{norme3}
\end{eqnarray}
Hence, in view of \eqref{truc} and \eqref{norme3}, we get the existence of some constant $C>0$ such that, for $t>0$ large enough,
\begin{equation}\label{esti2}
\mathcal E^{\rm r}_{\sqrt t,\theta(t)}\left(z(t,\cdot)\right)\leq -\frac{C}{t}+C\theta(t) t^{-\frac{1}{2}}.
\end{equation}

From \eqref{esti1} and \eqref{esti2}, we have, up to enlarging $C>0$, for $t>0$ large enough,
$$
t\mathcal E_{\sqrt t,\theta(t)}(z(t,\cdot))\leq  -C +\frac{C}{\sqrt t}+C\theta(t)t^{\frac 12}.
$$
From Assumption \ref{ASS-threshold} we know that $\liminf _{t\to +\infty}\theta(t)t^{\frac 12}=0$ and thus 
\begin{equation*}
\liminf _{t\to+\infty} t\mathcal E_{\sqrt t,\theta(t)}\left(z(t,\cdot)\right)<0,
\end{equation*}
which completes the proof of Claim \ref{claim1}.
\end{proof}

\noindent{\bf Second step: $u$ does converge to $1$ as $t\to+\infty$.}  We now complete the proof of Theorem \ref{theorem1} by showing that
$\lim_{t\to+\infty} u(t,x)=1$ locally uniformly for $x\in\R$.

To that aim let us fix  $0<\theta_0<\frac 12$, so that $\int_0^1 f(\theta_0,u)\d u>0$. It is known  \cite[Theorem 3.2]{Fife-Leod} that we can find $L>0$ large enough so that  the solution $v=v(t,x)$ of the initial value problem
\begin{equation}\label{fife}
\begin{cases}
\displaystyle \frac{\partial v}{\partial t}=\frac{\partial^2 v}{\partial x^2}+f\left(\theta_0,v\right),&t>0,\;x\in\R,\vspace{5pt}\\
\displaystyle v(0,x)=\frac{1+\theta_0}{2}\mathbf 1 _{(-L,L)}(x),&x\in\R,
\end{cases}
\end{equation} 
satisfies $v(t,x)\to 1$ as $t\to+\infty$, locally uniformly for $x\in\R$.

Next, according to the first step, there exists $x_0\in\R$ and a sequence $\{t_n\}_{n\geq 0}$ going to $+\infty$ as $n\to+\infty$ such that
\begin{equation*}
\liminf_{n\to+\infty} u(t_n,x_0)>0.
\end{equation*}
Consider the sequence of functions $u_n(t,x):=u(t+t_n,x)$. Because of parabolic regularity, one may assume that $u_n(t,x)\to u_\infty(t,x)$ locally uniformly for $(t,x)\in\R^2$ as $n\to+\infty$ and, since $\theta(t)\to 0$ as $t\to+\infty$, the function $u_\infty$ has to be an entire solution of the following monostable problem
\begin{equation*}
\frac{\partial u_\infty}{\partial t}=\frac{\partial^2 u_\infty}{\partial x^2}+u_\infty^2 \left(1-u_\infty\right),\;(t,x)\in\R^2,
\end{equation*} 
together with $u_\infty(0,x_0)>0$. According to the Hair Trigger Effect result of Aronson and Weinberger  \cite{Aronson-Weinberger}, one knows that $u_\infty(t,x)\to 1$ as $t\to+\infty$, locally uniformly for $x\in\R$.  Therefore, there exists $T>0$ large enough such that
\begin{equation*}
u_\infty(T,x)\geq \frac{3+\theta_0}{4},\;\;\forall x\in [-L,L].
\end{equation*} 
As a result, there is $n_0\geq 0$ large enough so that
\begin{equation*}
u(T+t_{n_0}, x)\geq \frac{1+\theta_0}{2}=v(0,x),\;\;\forall x\in [-L,L].
\end{equation*} 
Also, up to enlarging $n_0$, we have $\theta(t)\leq \theta_0$, so that  $f(\theta(t),\cdot)\geq f(\theta _0,\cdot)$, for all $t\geq t_{n_0}$. The parabolic comparison principle yields
\begin{equation*}
v(t,x)\leq u(T+t_{n_0}+t, x),\;\forall t\geq 0,\;x\in\R.
\end{equation*}
Because of the choice of the function $v$ and since $u\leq 1$, one concludes that
\begin{equation*}
\lim_{t\to+\infty} u(t,x)=1\text{ locally uniformly for $x\in\R$}.
\end{equation*}
This completes the proof of Theorem \ref{theorem1}.
\end{proof}

\section{Evolutionary rescue}\label{s:rescue}

In this section, we first prove the existence of a solution for the Cauchy problem \eqref{pb}--\eqref{pb1}, namely Theorem \ref{theo-existence}, and
then prove the evolutionary rescue phenomenon, namely Theorem \ref{theo-asymp}.

\subsection{Existence of a solution}

\begin{proof}[Proof of Theorem \ref{theo-existence}] We fix an initial data $(u_0(x),a_0)$ as in \eqref{pb1}. Let us observe that the main difficulty arises due to the gene flow term, more precisely the term $\frac{\partial \ln (u)}{\partial x}=\frac{1}{u}\frac{\partial u}{\partial x}$ in the $a-$equation. Indeed, despite the solution of the $u-$equation becomes immediately strictly positive for $t>0$, this term may become singular as $t\to 0$ in particular when $u_0$ vanishes at some points. To overcome this we make use of a regularisation procedure and we consider  \eqref{pb} with a sequence of positive initial data $(u_0^n(x),a_0)$. Next we pass to the limit $n\to+\infty$ to recover a solution of \eqref{pb}--\eqref{pb1} and complete the proof of Theorem \ref{theo-existence}. Due to the aforementioned singular term in the $a-$equation, the main difficulty in this proof consists in dealing with the initial data for $a$. Let us make this sketch precise. 

For $n\geq 1$ we define
\begin{equation*}
u_0^n(x):=\max\left(u_0(x),\frac{1}{n}\right).
\end{equation*}
Since $u_0^{n}\geq \frac 1 n$, we are equipped with  $(u^n,a^n)=(u^n(t,x),a^n(t,x))$ a classical solution of \eqref{pb} starting from initial data $(u_0^n(x),a_0)$. Let us first notice that $0<u^n(t,x)\leq 1$ and $0<a^{n}(t,x)\leq a_0$ for all $t\geq 0$, $x\in \R$. Also, the comparison principle ensures the following lower bound for $u^n$:
\begin{equation*}
u^n(t,x)\geq e^{-a_0 t} U_n(t,x),\;t\geq 0,\,x\in\R,\;n\geq 1,
\end{equation*}
wherein $U_n$ denotes the solution of the following heat equation
\begin{equation*}
\frac{\partial U_n}{\partial t}=\frac{\partial^2 U_n}{\partial x^2},\;t>0,\;x\in\R;\;\; U_n(0,\cdot)=u_0^n.
\end{equation*}
Note that $U_n(t,\cdot)=\Gamma(t,\cdot)\ast u_0^n$, where $\Gamma$ denotes the heat kernel. Since $u_0^n\to u_0$ in $L^\infty(\R)$, we have 
$U_n(t,\cdot)\to U_\infty(t,\cdot):=\Gamma(t,\cdot)*u_0$ in $L^\infty(\R)$ uniformly for $t>0$. Here since $u_0\not\equiv 0$, one has $U_\infty(t,x)>0$ for all $t>0$ and $x\in\R$. As a result, on any compact set $K$ of $(0,+\infty)\times \R$, there is $C_K>0$ such that, for all $n\geq 1$, all $(t,x) \in K$,  $u^n(t,x)\geq C_K$.

From standard parabolic estimates, the sequence $\{u^n\}$ is relatively compact in  $C_{\rm loc}\left([0,+\infty)\times\R\right)$ and in $C^{1+\frac \alpha 2, 2+\alpha}_{\rm loc}\left((0,+\infty)\times \R\right)$ for any $\alpha\in (0,1)$.
Next, because of the above lower bound for $u^n$, $\frac1{u_n}$ is uniformly bounded on each compact set of $(0,+\infty)\times \R$. As a consequence of standard parabolic estimates for the $a^n-$equation, the sequence $\{a^n\}$ is also relatively compact in  $C^{1+\frac \alpha 2, 2+\alpha}_{\rm loc}\left((0,+\infty)\times \R\right)$ for any $\alpha\in (0,1)$. Therefore one has, possibly along a subsequence, $u^n\to u$ as $n\to+\infty$ for the topologies of $C_{\rm loc}\left([0,+\infty)\times\R\right)$ and $C^{1+\frac \alpha 2, 2+\alpha}_{\rm loc}\left((0,+\infty)\times \R\right)$, while $a^n\to a$ as $n\to+\infty$ for the topology of $C^{1+\frac \alpha 2, 2+\alpha}_{\rm loc}\left((0,+\infty)\times \R\right)$. Furthermore the limit functions $u$ and $a$ satisfy the following properties
\begin{equation*}
\begin{split}
&e^{-a_0 t}U_\infty(t,x)\leq u(t,x)\leq 1,\;\;0\leq a(t,x)\leq a_0,\\
&u(t,x)\to u_0(x)\text{ locally uniformly for $x\in\R$ as $t\to 0$},\\
&\dfrac{\partial u}{\partial t}=\dfrac{\partial^2 u}{\partial x^2}+u\left(u-a^2\right)(1-u),\;t>0,\;x\in\R,\\
&\displaystyle \frac{\partial a}{\partial t}=\frac{\partial^2 a}{\partial x^2}+2\frac{\partial a}{\partial x}\frac{\partial \ln(u)}{\partial x}-\varepsilon (1-u)a,\;t>0,\;x\in\R.\\
\end{split}
\end{equation*}

To complete the proof of our existence result, it remains to show that $\Vert a(t,\cdot)-a_0\Vert _\infty =O(t)$ as $t\to 0$. To that aim, we consider the auxiliary sequence of functions $v^n(t,x):=a^n(t,x)u^n(t,x)$. Observe that $v^n$ satisfies
\begin{equation*}
\begin{cases}
\dfrac{\partial v^n}{\partial t}=\dfrac{\partial^2 v^n}{\partial x^2}+v^n G_n(t,x),&\;t>0,\;x\in\R,\vspace{5pt}\\
v^n(0,x)=a_0 u_0^n(x),&\; x\in \R,
\end{cases}
\end{equation*}
wherein we have set 
\begin{equation*}
G_n(t,x)=-\varepsilon \left(1-u^n(t,x)\right)+\left(u^n(t,x)-a^n(t,x)^2\right)\left(1-u^n(t,x)\right).
\end{equation*}
Since $u^n$ and $a^n$ are uniformly bounded, there exists $M>0$ such that
\begin{equation*}
|G_n(t,x)|\leq M,\;\forall (t,x)\in [0,+\infty)\times\R,\;\forall n\geq 1.
\end{equation*}
As a consequence of the parabolic comparison principle, one obtains 
\begin{equation*}
a_0 U_n(t,x) e^{-M t}\leq v^n(t,x)\leq a_0 U_n(t,x) e^{Mt},\;\forall (t,x)\in [0,+\infty)\times\R,\;\forall n\geq 1.
\end{equation*}
Similarly, there  exists some constant $N>0$ such that
\begin{equation*}
U_n(t,x) e^{-N t}\leq u^n(t,x)\leq U_n(t,x) e^{Nt},\;\forall (t,x)\in [0,+\infty)\times\R,\;\forall n\geq 1.
\end{equation*}
Since $a^n=\frac{v^n}{u^n}$, we infer from the two above estimates that
\begin{equation}\label{ici}
a_0 e^{-(M+N)t}\leq a^n(t,x)\leq a_0 e^{(M+N)t},\;\forall (t,x)\in [0,+\infty)\times\R,\;\forall n\geq 1.
\end{equation}
Passing to the limit $n\to +\infty$ implies that the function $a$ satisfies
\begin{equation*}
a_0 e^{-(M+N)t}\leq a(t,x)\leq a_0 e^{(M+N)t},\;\forall (t,x)\in (0,+\infty)\times\R,
\end{equation*}
which is enough to complete the proof of the theorem.
\end{proof}

\begin{remark} If the initial data $a_0(x)$ is non constant and we try to reproduce the above argument, then \eqref{ici} is replaced by
$$
\frac{V_n(t,x)}{U_n(t,x)}e^{-(M+N)t}\leq a^n(t,x)\leq \frac{V_n(t,x)}{U_n(t,x)}e^{(M+N)t},
$$
where $V_n(t,x)$ solves the heat equation starting from $a_0(x)u_0^n(x)$. After letting $n\to+\infty$, we get
$$
\frac{\left(\Gamma(t,\cdot)*(a_0u_0)\right)(x)}{\left(\Gamma(t,\cdot)*u_0\right)(x)}e^{-(M+N)t}\leq a(t,x)\leq \frac{\left(\Gamma(t,\cdot)*(a_0u_0)\right)(x)}{\left(\Gamma(t,\cdot)*u_0\right)(x)}e^{(M+N)t},
$$
for all $(t,x)\in(0,+\infty)\times \R$, and this is not clear that one recovers the initial data $a_0(x)$ for $a(t,x)$, in particular in the points where $u_0$ vanishes.
\end{remark}

\subsection{Evolutionary rescue result}

\begin{proof}[Proof of Theorem \ref{theo-asymp}] The proof is an application of Theorem \ref{theorem1}.
We fix $u_0\in C(\R)\setminus\{0\}$ with $0\leq u_0\leq 1$ and $a_0\in (0,1)$. Let $(u,a)=(u(t,x),a(t,x))$ be a solution of \eqref{pb}, with initial data $(u_0(x),a_0)$, in the sense of Definition \ref{def-sol}.

Let us first show that $u$ satisfies the weak persistence property, namely
\begin{equation}\label{1}
\limsup_{t\to+\infty}\sup_{x\in\R} u(t,x)=1.
\end{equation}
To that aim we argue by contradiction by assuming that there exist $t_0>0$ and $\alpha\in (0,1)$ such that
\begin{equation*}
u(t,x)\leq \alpha,\;\forall t\geq t_0,\;x\in\R.
\end{equation*}
Hence, the $a-$equation  yields
$$
\frac{\partial a}{\partial t}\leq \frac{\partial^2 a}{\partial x^2}+2\frac{\partial a}{\partial x}\frac{\partial \ln(u)}{\partial x}-\varepsilon (1-\alpha)a,\;\; \forall t\geq t_0,\;x\in\R,
$$
and we deduce from the comparison principle that
\begin{equation*}
a(t,x)\leq a_0 e^{-\varepsilon (1-\alpha)(t-t_0)},\;\;\forall t\geq t_0,\;x\in\R.
\end{equation*}
Hence setting $\theta(t):= a_0^2 e^{-2\varepsilon (1-\alpha)(t-t_0)}$, one obtains from the comparison principle applied to the $u-$equation that
\begin{equation*}
u(t,x)\geq \underline u(t,x),\;\forall t\geq t_0,\;x\in\R,
\end{equation*}
with $\underline u=\underline u(t,x)$ the solution of the Cauchy problem
\begin{equation*}
\begin{cases}
\displaystyle \frac{\partial \underline u}{\partial t}=\frac{\partial^2 \underline u}{\partial x^2}+f(\theta(t),\underline u),&\;t\geq t_0,\;x\in\R,\vspace{5pt}\\
\displaystyle \underline u(t_0,x)=u(t_0,x) &\;x\in\R.
\end{cases}
\end{equation*}
The function $\theta$ decays exponentially to $0$ and clearly satisfies Assumption \ref{ASS-threshold}. Moreover 
$u(t_0,\cdot)\in C(\R) \setminus \{0\}$ and $0\leq u(t_0,\cdot)\leq 1$, so that  Theorem \ref{theorem1} applies for the above equation and ensures that $\underline u(t,x)\to 1$ locally uniformly for $x\in\R$ as $t\to+\infty$. This contradicts $1>\alpha\geq u\geq\underline u$ for large times and \eqref{1} follows.

The proof of the second statement in Theorem \ref{theorem1} is similar to the one of the second step in the proof of Theorem \ref{theorem1}. Since $ a_0^{2}<\frac 12$  we can find $L>0$ large enough so that, for any point $x^{*}\in\R$,  the solution $v=v(t,x;x^*)$ of the initial value problem 
\begin{equation}\label{fife2}
\begin{cases}
\displaystyle \frac{\partial v}{\partial t}=\frac{\partial^2 v}{\partial x^2}+f\left(a_0^{2},v\right),&t>0,\;x\in\R,\vspace{5pt}\\
\displaystyle v(0,x;x^*)=\frac{1+a_0^2}{2}\mathbf 1 _{(x^{*}-L,x^{*}+L)}(x),&x\in\R,
\end{cases}
\end{equation} 
satisfies $v(t,x)\to 1$ as $t\to+\infty$, locally uniformly for $x\in\R$. From \eqref{1}, there are $t_n\to +\infty$ and $x_n\in \R$ such that  $u(t_n,x_n)\to 1$. 
Consider the sequences of functions
$$
u_n(t,x):=u(t+t_n,x+x_n),\; a_n(t,x):=a(t+t_n,x+x_n).
$$
From parabolic regularity and the uniform boundeness of $a_n$, we can let $n\to +\infty$ and get an entire (weak) solution of 
\begin{equation*}
\frac{\partial u_\infty}{\partial t}=\frac{\partial^2 u_\infty}{\partial x^2}+u_\infty(u_\infty-a_\infty(t,x)) \left(1-u_\infty\right),\;(t,x)\in\R^2,
\end{equation*} 
together with $u_\infty(0,0)=1$ and $0\leq u_\infty \leq 1$. This enforces $u_\infty \equiv 1$. Hence $u_n(t,x)\to 1$ locally uniformly for $t\geq 0$, $x\in \R$. Therefore, there is $n_0\geq 0$  large enough so that
\begin{equation*}
u(t_{n_0},\cdot)\geq \frac{1+a_0^2}{2}\mathbf 1 _{(x_{n_0}-L,x_{n_0}+L)}=v(0,x;x_{n_0}).
\end{equation*} 
Since  $f(a^2(t,x),\cdot)\geq f(a_0^2,\cdot)$, the parabolic comparison principle yields
\begin{equation*}
v(t,x;x_{n_0})\leq u(t_{n_0}+t, x),\;\forall t\geq 0,\;x\in\R.
\end{equation*}
Because of the choice of the function $v$ and since $u\leq 1$, one concludes that
$\lim_{t\to+\infty} u(t,x)=1$  locally uniformly for $x\in\R$.
This completes the proof of Theorem \ref{theo-asymp}.
\end{proof}

\medskip

\noindent {\bf Acknowledgement.} The first author is supported by the I-site MUSE, project MICHEL.

\end{document}